\documentclass{amsart}

\DeclareSymbolFont{bbold}{U}{bbold}{m}{n}
\DeclareSymbolFontAlphabet{\mathbbold}{bbold}

\usepackage{amsmath}%

\usepackage{amsfonts}%
\usepackage{amssymb}%
\usepackage{graphicx}
\usepackage{bm}

\usepackage{mathabx}
\usepackage{rotating}
\usepackage[all,cmtip]{xy}
\usepackage{cite}
\usepackage{hyperref}
\usepackage{geometry}
\usepackage{bbm}
\usepackage{stmaryrd}
\usepackage{tensor}
\usepackage{wasysym}

\makeatletter
\newcommand{\colim@}[2]{%
  \vtop{\m@th\ialign{##\cr
    \hfil$#1\operator@font colim$\hfil\cr
    \noalign{\nointerlineskip\kern1.5\ex@}#2\cr
    \noalign{\nointerlineskip\kern-\ex@}\cr}}%
}
\newcommand{\colim}{%
  \mathop{\mathpalette\colim@{\rightarrowfill@\scriptscriptstyle}}\nmlimits@
}
\renewcommand{\varprojlim}{%
  \mathop{\mathpalette\varlim@{\leftarrowfill@\scriptscriptstyle}}\nmlimits@
}
\renewcommand{\varinjlim}{%
  \mathop{\mathpalette\varlim@{\rightarrowfill@\scriptscriptstyle}}\nmlimits@
}
\makeatother

\newtheorem{dummy}{dummy}[section]              
\newtheorem{lemma}[dummy]{Lemma}
\newtheorem{theorem}[dummy]{Theorem}
\newtheorem{corollary}[dummy]{Corollary}
\newtheorem{proposition}[dummy]{Proposition}

\theoremstyle{definition}                                  
\newtheorem{definition}[dummy]{Definition}

\newtheorem{remark}[dummy]{Remark}

\newcommand{\Vect}{\mathbf{Vect}}

\DeclareMathOperator{\Hom}{Hom}

\DeclareMathOperator{\End}{End}

\newcommand{\module}{\mathrm{mod}}

\newcommand{\QC}{QC}

\DeclareMathOperator{\Rep}{Rep}

\newcommand{\Cat}{\mathbf{DGCat}}

\newcommand{\BAR}{\mathbf{Bar}}
\newcommand{\Tot}{\mathbf{Tot}}

\newcommand{\bs}{\backslash}

\newcommand{\ls}[2]{\tensor*[^{#1 }]{{#2}}{}}
\newcommand{\lrsub}[3]{\tensor*[_{#1}]{{#2}}{_{#3}}}

\DeclareMathOperator{\Ind}{Ind}

\newcommand{\HC}{HC}

\newcommand{\Spr}{Spr}

\newcommand{\fg}{\mathfrak{g}}

\newcommand{\cC}{\mathcal C}
\newcommand{\cD}{\mathcal D}

\newcommand{\cH}{\mathcal H}

\newcommand{\cN}{\mathcal N}

\newcommand{\cS}{\mathcal S}

\newcommand{\cU}{\mathcal U}

\newcommand{\wt}{\widetilde}

\makeatletter
\newcommand*\leftdash{\rotatebox[origin=c]{-45}{$\dabar@\dabar@\dabar@$}}
\newcommand*\rightdash{\rotatebox[origin=c]{45}{$\dabar@\dabar@\dabar@$}}
\makeatother

\newcommand{\quot}[3]{{#1}\backslash{#2}/{#3}} 
\newcommand{\wq}[2]{{#1}\leftdash{#2}}

\title{Highest Weights for Categorical Representations}

\author{David Ben-Zvi} \address{Department of Mathematics\\University
  of Texas\\Austin, TX 78712-0257} \email{benzvi@math.utexas.edu}
\author{Sam Gunningham} \address{Department of Mathematics\\University
  of Texas\\Austin, TX 78712-0257} \email{gunningham@math.utexas.edu}
\author{Hendrik Orem}

\begin{document}

\begin{abstract}
We present a criterion for establishing Morita equivalence of monoidal categories, and apply it to the categorical representation theory
of reductive groups $G$. We show that the ``de Rham group algebra" $\cD(G)$ (the monoidal category of $\cD$-modules on $G$) is Morita equivalent to the universal Hecke category $\cD(\quot{N}{G}{N})$ and to its monodromic variant $\wt{\cD}(\quot{B}{G}{B})$. In other words, de Rham $G$-categories, i.e., module categories for $\cD(G)$, satisfy a ``highest weight theorem" - they all appear in the decomposition of the universal principal series representation $\cD(G/N)$ or in twisted $\cD$-modules on the flag variety $\wt{\cD}(G/B)$. 
\end{abstract}
\maketitle

\section{Introduction}

\subsection{Morita equivalences}
The fundamental building block for representations of reductive groups is the principal series, the collection of representations constructed by parabolic induction from a maximal torus. These representations can be described as the constituents of the appropriate space of functions on the coset space $G/N$ (where $N$ is the unipotent radical of a Borel subgroup $B$), and are parametrized by modules for the Hecke algebra of double cosets $\quot{N}{G}{N}$. In the setting of finite dimensional representations of compact or complex reductive groups, the Borel-Weil theorem asserts that all representations are principal series so that we obtain a complete classification of representations based on their highest weights, as representations of the torus $H=B/N$, and their symmetries, identified with the Weyl group. The Casselman subrepresentation theorem provides an analogous picture for real reductive groups, showing that any irreducible admissible representation arises as a subrepresentation of a principal series representation. This assertion fails if we keep track of unitary structure, for example due to existence of discrete series. Likewise for representations of finite and p-adic reductive groups, analogous statements fail due to the existence of cuspidal representations. Our goal in this paper is to prove the corresponding result for the categorical representations of a complex reductive group $G$: they all appear in the principal series, so that there are no cuspidal categorical representations. 

Just as representations of a group $G$ arise naturally as spaces of functions on $G$-spaces, categorical representations of an algebraic group $G$ arise naturally as categories of sheaves on $G$-varieties $X$. Such categories are automatically module categories for the corresponding ``group algebra", the monoidal category of sheaves on $G$ equipped with the convolution product. (In order to perform ``functional analytic" operations on categories of sheaves it is essential to live in an appropriate setting of homotopical algebra, which for us will be the $\infty$-category of differential graded categories.) 

Two natural classes of sheaves to consider are quasi coherent sheaves and algebraic $\cD$-modules. Throughout, given a stack $X$, we denote by $QC(X)$ and $\cD(X)$ the $\infty$-categories of quasi coherent sheaves and $\cD$-modules on $X$ \cite{GR}. Thus we arrive at the notions of algebraic (or weak) and de Rham (or strong) categorical representations of $G$. These are module categories for the ``algebraic group algebra" $QC(G)$ or the ``de Rham group algebra" $\cD(G)$. The former (weak) notion is simpler -- in particular the definitive work of Gaitsgory \cite{1affine} shows algebraic $G$-categories are generated by the {\em trivial} module, $\Vect = \QC(pt)$. 
The latter (strong) notion, whose study was pioneered in \cite{BD Hecke}, is the most relevant to geometric representation theory. In particular the adjoint action of $G$ on $\fg$ endows the category $\cU(\fg)-\module$ with a de Rham $G$-action; it was shown by Beraldo (using Gaitsgory's 1-affineness result) that de Rham $G$-categories are generated by the category of $\cU(\fg)$-modules. Beilinson-Bernstein localization can be seen as a relation between this representation and the geometric representation given by twisted $\cD$-modules on the flag variety, or equivalently the ``monodromic principal series" category $\cD(G/N)_H$ of $H$-monodromic (i.e., weakly $H$-equivariant) $\cD$-modules on $G/N$. 

To explain the utility of having generators for the category of $G$-categories, recall the following fundamental result of Morita theory: if an abelian category $\cC$ has a compact projective generator $M$, then $\cC$ is equivalent to the category of modules for the corresponding Hecke algebra $\cH(M) := \End_\cC(M)$. The theorem of Gaitsgory asserts that the trivial algebraic $G$-category, $\Vect$ behaves as if it were a projective generator of the category of algebraic $G$-categories; in particular, the category of algebraic $G$-categories is equivalent to modules for the monoidal category $\Rep(G) = \End_{QC(G)}(\Vect)$. This situation is as nice as possible: the monoidal category $\Rep(G)$ is rigid, symmetric, and semisimple, making it amenable to study by algebraic and combinatorial methods. Similarly, Beraldo's result asserts that de Rham $G$-categories are equivalent to module categories for the monoidal category $\HC_G = \End_{\cD(G)}(\cU(\fg)-\module)$ of Harish-Chandra bimodules. The monoidal category $\HC_G$ is rigid (although it is not symmetric and it is far from being semisimple).

Our main result asserts that all de Rham $G$-categories are generated by either the principal series category $\cD(G/N)$ or its monodromic version $\cD(G/N)_H$.

\begin{theorem}
\label{thm:main}
Let $G$ be a complex reductive group, with Borel $B$, $N= [B,B]$, and $H=B/N$. The $\cD(G)-\cD(\quot NGN)$-bimodule $\cD(G/N)$ defines a Morita equivalence between the monoidal categories $\cD(G)$ and $\cD(\quot NGN)$.
\end{theorem}

\begin{corollary}
The following monoidal categories are all Morita equivalent.
\begin{enumerate}
\item $\cD(G)$,
\item $\cD(\quot{N}{G}{N})$,
\item $\cD_{H}(\quot{N}{G}{N})_H$,
\item $\HC_G$.
\end{enumerate}
\end{corollary}

%

To prove Theorem \ref{thm:main}, we will formulate a more general statement giving sufficient conditions for when a module category $M$ for a monoidal category $A$ defines a Morita equivalence between $A$ and $\cH(M):=\End_A(M)$.

For the purposes of analogy, let us consider the following situation. Let $A$ be a finite dimensional algebra over a field $k$, and $M$ a right $A$-module.  Note that $M$ is naturally a $\cH(M)-A$-bimodule, so determines a functor
\[
M \otimes_A (-): A-\module \to \cH(M)-\module
\]
Now suppose that $M$ is projective and faithful as an $A$-module. In that case the module $M$ defines a Morita equivalence between $A$ and $\cH(M)$, which is to say, the functor $M \otimes_A (-)$ defines an equivalence between $A-\module$ and $\cH(M)-\module$.

Now let $A$ and $B$ be monoidal $\infty$-categories, and $M = {}_B M_A$ a bimodule category. As above, we think of $M$ as defining a functor:
\[ 
M\otimes_A (-): A-\module \to B-\module.
\]
We say that $M$ has a left (respectively right) dual if the corresponding functor has a left (respectively right) adjoint. A bimodule $M$ is \emph{proper dualizable} if $M$ is dualizable as an $A$-module, and the evaluation and coevaluation maps admit continuous right adjoints (see Definition \ref{def:properlyDualizable} for the full definition).

Our key technical result is the following 
\begin{theorem}
\label{thm:morita}
Suppose $M = \lrsub BMA$ is a proper dualizable bimodule. Then there is a fully faithful embedding
\[
i:\cH(M)-\module \hookrightarrow A-\module.
\]
If, in addition, the action map $A \to \End_B(M)$ is conservative then $i$ is an equivalence of categories.
\end{theorem}

We will apply this result in the case $A=\cD(G)$, $B=\cD(H)$, and $M=\cD(N\bs G)$ to obtain the equivalence between $\cD(G)-\module$ and $\cD(\quot NGN)-\module$ in Theorem \ref{thm:main}.

\subsection{Acknowledgments}
We would like to acknowledge the National Science Foundation for its support through individual grant DMS-1103525 (DBZ).

\section{Abstract Morita Theory}

\subsection{$\infty$-categorical preliminaries}
Throughout, we rely on the foundations developed in \cite{l-ha}. Let $k$ be a field, and let $\Cat$ denote the $(\infty,1)$ category of $k$-linear, stable, presentable $\infty$-categories where morphisms are functors which are left adjoints (equivalently, functors which preserve small colimits). Recall that $\Cat$ comes equipped with a symmetric monoidal structure $\otimes$. The unit object of $\Cat$ is the category $S=\Vect_k$ of (differential graded) $k$-vector spaces.

Recall that a pair of adjoint functors:
\[
G: C \leftrightarrows D: F
\]
(where $G$ is left adjoint to $F$) give rise to a monad $FG$ acting on $C$. We will denote by $C^{FG}$ the category of modules for $FG$ in $C$.

\begin{theorem}[Barr-Beck-Lurie, easy version \cite{l-ha}, Corollary 4.7.4.16]
Let $F:D\to C$ be a morphism in $\Cat$ (in particular, $F$ preserves colimits) which admits a left adjoint $G$. Then $F$ factors canonically as 
\[
\xymatrix{
D \ar[rd]_{\overline{F}} \ar[rr]^F &&  C\\
& C^{FG} \ar[ru] &
}
\]
where $\overline{F}$ has a fully faithful left adjoint. If, in addition, $F$ is conservative then $\overline{F}$ is an equivalence (in that case, we say that $F$ is monadic).
\end{theorem}

The following lemma states that certain colimits in $\Cat$ can be computed as limits of the corresponding diagram of right adjoints. The result appears as Lemma 1.3.3 in \cite{DGcat}; it is a consequence of Corollary 5.3.3.4 in \cite{topos}.
\begin{lemma}\label{lemma:limcolim}
Let $I$ be an $\infty$-category, and $I \to \Cat$, $i \mapsto C_i$ a functor. For each $\alpha: i \to j$, let $F_\alpha: C_i \to C_j$ denote the corresponding functor. Suppose each $F_\alpha$ admits a continuous right adjoint, and consider the corresponding diagram $I^{op} \to \Cat$ formed by taking the right adjoints of each $F_\alpha$. Then there is a canonical equivalence
\[
\xymatrix{
\colim_{i\in I}(C_i) \ar[r]^{\sim} & \varprojlim_{j\in I^{op}}(C_j).
}
\]
\end{lemma}

The theory of monadic descent in the context of $\infty$-categories was developed by Lurie in \cite{l-ha}, Chapter 4.7.6. The following definition and theorem are what we require in this paper.
\begin{definition}
Let $C^\bullet$ be a cosimplicial object (respectively, augmented cosimplicial object) of $\Cat$; we say that $C^\bullet$ satisfies the monadic Beck-Chevalley conditions, if the following two conditions hold:
\begin{enumerate}
\item For every object $[i]$ in the simplex category $\Delta$ (respectively the augmented simplex category  $\Delta^+$), the last face map $C^{\partial^0_i}:C^{i} \to C^{i+1}$ admits a left adjoint $(C^{\partial^0_i})^L$;
\item For every morphism $\alpha: [i] \to [j]$ in $\Delta$ (respectively $\Delta^+$), the diagram
\[
\xymatrix{
C^{i+1} \ar[r]^{(C^{\partial^0_i})^L}  & C^i\\
C^{j+1} \ar[r]_{(C^{\partial^0_i})^L} \ar[u]^{C^\alpha}& C^j \ar[u]_{C^\alpha},
}
\]
commutes.
\end{enumerate}
\end{definition}

\begin{theorem}[Monadic descent, \cite{l-ha}, Theorem 4.7.6.2, Corollary 4.7.6.3]\label{thm:md}
\begin{enumerate}
\item Let $\widetilde{C}^\bullet$ be an coaugmented cosimplicial object of $\Cat$ which satisfies the monadic Beck-Chevalley conditions. Let $C^\bullet$ denote the cosimplical object without the coaugmentation. Then the canonical map $\overline{F}: \widetilde{C}^{-1} \to \Tot(C^\bullet)$ admits a fully faithful right adjoint. If the coaugmentation $\widetilde{C}^{-1} \to C^0$ is conservative, then $\overline{F}$ is an equivalence.
\item Let $C^\bullet$ be a cosimplicial object of $\Cat$. Then the coaugmented diagram $\Tot(C^\bullet) \to C^\bullet$ satisfies the monadic Beck-Chevalley conditions.
\end{enumerate}
\end{theorem}

\subsection{Monoidal categories, modules, and dualizibility}
In this paper, a \emph{monoidal category} means an algebra object in $\Cat$, in the sense of \cite{l-ha}, Chapter 4. Given a monoidal category $A$, we have the notion of a (left or right) \emph{module category} $M$ for $A$ (\cite{l-ha}, Chapter 4.2, 4.3). In particular this means that $M$ and the structure maps $A\otimes M \to M$ are in $\Cat$. We write $A-\module$ for the category of left $A$-module objects in $\Cat$, and we identify the category of right $A$-modules with $A^{op}-\module$ (note that $A^{op}$ here refers to $A$ with the opposite monoidal structure, not the opposite category). Given monoidal categories $A$ and $B$, a $B-A$-bimodule is the same thing as an object of $B \otimes A^{op}-\module$.

Given an $A-A$-bimodule $M$, the Hochshild homology category $HH(A;M)$ can be identified with the colimit of the simplicial bar complex $\BAR_\bullet(A; M)$, whose $n$-simplices are given by $M \otimes A^{\otimes n}$. In the case where the bimodule $M$ is of the form $L \otimes N$, where $L$ is a left $A$-module, and $N$ is a right $A$-module, the Hochshild homology $HH(A; L\otimes N)$ can be identified with the relative tensor product $L \otimes_A N$. In that case we will often write $\BAR_\bullet(A;L,N)$ for the corresponding bar complex. 

Suppose $A$ and $B$ are monoidal categories, and $M$ a $B-A$-bimodule. Thus $M$ determines a functor 
\[
M \otimes_A (-): A-\module \to B-\module.
\]
\begin{definition}
We say that $M = \lrsub BMA$ is \emph{$A$-dualizable} (or \emph{left dualizable}) if $M \otimes_A (-)$ has a left adjoint. Similarly, $M$ is $B$-dualizable (or \emph{right dualizable}) if $M \otimes_A (-)$ has a right adjoint. 
\end{definition}

More explicitly, the bimodule $M$ is $A$-dualizable if and only if there exists an $A-B$-bimodule $\ls \vee M$, together with morphisms
\begin{align}
u^L:B \to M \otimes_A \ls \vee M, \\
c^L : \ls \vee M \otimes_B M \to A,
\end{align}
such that the \emph{triangle identities} hold, that is, the composite functors
\begin{align}\label{eq:zorro}
M \xrightarrow{u_L \otimes 1_M} M \otimes_A \ls \vee M \otimes_B M \xrightarrow{1_M \otimes c_L} M \\
\ls \vee M \xrightarrow{1_{\ls \vee M} \otimes u^L} \ls \vee M \otimes_B M \otimes _A \ls \vee M \xrightarrow{c^L \otimes 1_{\ls \vee M}} \ls \vee M
\end{align}
are equivalences.

Similarly, $M$ is $B$-dualizable if there exists an $A-B$-bimodule $M^\vee$, together with morphisms:
\begin{align}
c^R: M \otimes_A M^\vee \to B, \\
u^R: A \to M^\vee \otimes_B M,
\end{align}
such that the corresponding triangle identities hold.

\begin{remark}
\begin{enumerate}
\item Note that the condition that $M$ be $A$ dualizable (repectively $B$) depends only on the structure of $M$ as an $A$-module (respectively, a $B$-module). To see this, note that the data $(u^L, c^L)$ is equivalent to data:
\begin{align}
u_0^L: S \to M \otimes_A \ls \vee M, \\
c^L : \ls \vee M \otimes_S M \to A,
\end{align}
satisfying the triangle condition (recall that $S=\Vect_k$ denotes the unit of $\Cat$). In other words, the bimodule ${}_B M_A$ being left dualizable is equivalent to ${}_S M_A$ being left dualizable.
\item If $M$ is left dualizable, then for any right $A$-module $N$, we have
\[
\Hom_{A^{op}}(M,N) \simeq N \otimes_A \ls\vee M.
\]
(Here $A^{op}$ refers to $A$ with the opposite \emph{monoidal} structure).
\item If $M$ is left dualizable, we have an identification
\[
\ls \vee M \otimes _A M \simeq \Hom_A(M, M).
\]
The unit morphism $u^L$ becomes identified with the inclusion of the unit object in the monoidal category $S \to \cH(M) = \Hom_A(M,M)$.
\item A sufficient condition for $M$ to be dualizable as an $S$-module is that it is compactly generated. In that case $M^\vee = \Ind(M_{c}^\diamondsuit)$, where $M^\diamondsuit_{c}$ means the opposite category of the subcategory of compact objects.
\end{enumerate}
\end{remark}

\begin{definition}
\label{def:properlyDualizable}
We say that the $B-A$-bimodule $M = {}_B M_A$ is \emph{properly dualizable} if $M$ is $A$-dualizable and the evaluation and coevaluation morphisms have continuous right adjoints.
\end{definition}

\begin{remark}
Suppose $F:A\to B$ is a morphism in $\Cat$ where $A,B$ are compactly generated. Then $F$ has a continuous right adjoint if and only if $F$ sends compact objects to compact objects.
\end{remark}

\begin{lemma}\label{lemmaleftright}
Suppose $M= {}_B M_A$ is a properly dualizable bimodule. Then $M$ is $B$-dualizable. Moreover, the $B$-dual $M^\vee$ is equivalent to the $A$-dual $\ls \vee M$, and the unit (respectively counit) morphism for the right duality is given by the right adjoint of the counit (respectively unit) for the left duality.
\end{lemma}
\begin{proof}
As $M$ is left dualizable, we have functors:
\begin{align}
u^L:B \to M \otimes_A \ls \vee M, \\
c^L : \ls \vee M \otimes_B M \to A,
\end{align}
such that the triangle identities of Equation~\ref{eq:zorro} are satisfied.
Taking right adjoints, we obtain functors:
\begin{align}
c^R: M \otimes_A \ls \vee M \to B, \\
u^R : A \to \ls \vee M \otimes_B M,
\end{align}
where $u^R$ is the right adjoint of $c^L$ and $c^R$ is the right adjoint of $u^L$. By the assumption of proper dualizability, $u^R$ and $c^R$ are continuous, i.e., morphisms in $\Cat$. Moreover, we observe that the functors $(c^R, u^R)$ satisfy the triangle identities for $^\vee M$ to be the right dual to $M$, by taking the right adjoints of the triangle identites for the left duality $(c^L, u^L)$.
\end{proof}

\subsection{Rigid monoidal categories}
A monoidal category $A$ is called rigid if it is compactly generated, and every compact object is (left and right) dualizable.\footnote{There is a more general definition of rigid monoidal category which does not require compact generation, see \cite{1affine}, D.1.1. The definitions agree in the compactly generated case.} Given a rigid monoidal category $A$, let $A'$ denote its dual as an object of $\Cat$ (recall that this can be constructed as the Ind-category of the opposite category of compact objects of $A$). The operations taking left and right duals define functors
\[
\phi^R, \phi^L: A \to A^\vee
\]
The composition $\varphi: (\phi^R )^{-1}\phi^L$ is a monoidal autoequivalence of $A$. A \emph{pivotal} structure on $A$ is a monoidal equivalence of functors $\alpha: \varphi \simeq id_A$.

\begin{theorem}[Gaitsgory, \cite{1affine} Appendix D]\label{thm:rigid}
Let $A$ be a rigid monoidal category.
\begin{enumerate}
\item A left $A$-module $L$ is dualizable if and only if $L$ is dualizable as an object of $\Cat$. 
\item If $A$ is equipped with a pivotal structure and $M$ is an $A$-bimodule, there is a canonical identification
\[
HH_\ast (A;M) \simeq HH^\ast(A;M).
\]
\end{enumerate}
\end{theorem}
\begin{remark}
Part 2 of Theorem \ref{thm:rigid} is proved by observing that the structure morphisms in the simplicial bar object computing $HH_\ast(A;M)$ have continuous right adjoints, and the corresponding diagram of right adjoints is identified with the cosimplicial cobar object computing $HH^\ast(A;M)$. The identification of $HH_\ast(A;M)$ and $HH^\ast(A;M)$ then follows from Lemma \ref{lemma:limcolim}.
\end{remark}

\section{Proof of Theorem \ref{thm:morita}}
\label{sec:proof}

Fix monoidal categories $A$ and $B$, and properly dualizable bimodule $M = {}_B M_A$ with dual $\ls \vee M = \Hom_{A^{op}}(M,A)$. 

Note that the Hecke algebra $\cH := \cH(M) := \End_A(M)$ is identified with $M \otimes_A \ls \vee M$, by $A$-duality. There is an algebra morphism $B \to \cH$ given by the action of $B$ on $M$. Thus we may consider $\cH$ as an algebra object in $B-B$-bimodules. The $B-A$-bimodule $M$ carries a canonical left action of $\cH(M)$ which commutes with the $A$-action, thus promoting $M$ to a $\cH(M)-A$-bimodule; similarly, $\ls \vee M$ is naturally a $A-\cH(M)$ bimodule.

To compute the relative tensor product $\ls \vee M \otimes_\cH M$, we will use the bar construction of $\cH$ considered as an algebra object in the monoidal category of $B-B$-bimodules. Explicitly, we define the simplicial category $C_\bullet$ with simplices
\begin{align*}
C_n &= \ls \vee M \otimes_B \cH^{\otimes_B n} \otimes_B M\\
&\simeq \ls \vee M \otimes_B ( M \otimes_A \ls \vee M )^{\otimes_B n}\otimes _B M \\
& \simeq (\ls \vee M \otimes_B M)^{\otimes_A (n+1)}.
\end{align*}
The face maps $C_{\partial_n^i}: C_{n} \to C_{n-1}$ are given by inserting $c^L: \ls \vee M \otimes_B M \to A$ in the $i$th tensor factor of 
\[
C_n \simeq  (\ls \vee M \otimes_B M)^{\otimes_A (n+1)}, 
\]
and the degeneracy maps $C_{e_n ^j}: C_n \to C_{n+1}$ are given by inserting $u^L: B \to M \otimes_B \ls \vee M$ after the $j$th tensor factor of
\[
 C_n \simeq  \ls \vee M \otimes_B ( M \otimes_A \ls \vee M )^{\otimes_B n}\otimes _B M.
\]
If we set $C_{-1} = A$, the morphism $c^L:C_0 \to C_{-1}$ defines an augmentation of $C_\bullet$ defining an augmented simplicial object $\widetilde{C}_\bullet$. As the functors $c^L$ and $u^L$ have continuous right adjoints, it follows that the structure maps of $C_\bullet$ have continuous right adjoints, and thus the cosimplicial category $C^\bullet$ obtained by taking right adjoints is a cosimplicial object in $\Cat$.  
\begin{lemma}\label{lemma:bc}
The coaugmented cosimplicial category $\widetilde{C}^\bullet$ satisfies the monadic Beck-Chevalley conditions.
\end{lemma}
\begin{proof}
We should check that, for each $\alpha: [i] \to [j]$ in $\Delta^+$, the square
\[
\xymatrix{
C^{i+1} \ar[r]^{(C^{\partial^0_i})^L}  & C^i \\
C^{j+1} \ar[r]_{(C^{\partial^0_i})^L} \ar[u]^{C^\alpha}& C^j \ar[u]_{C^\alpha},
}
\]
commutes.
For example, in the case when $\alpha$ is the unique map $[-1] \to [0]$, we have the diagram
\[
\xymatrix{
\End_B(M) \otimes_A \End_B(M) \ar[rrr]^{id \otimes c^L} &&& \End_B(M) \\
 \ar[rrr]^{c^L}  \End_B(M) \ar[u]^{u^R \otimes id} &&& A \ar[u]^{u^R}\\
}
\]
which commutes by inspection. The remaining cases are similar. 
\end{proof}

Note that we have a commutative diagram of monoidal functors:
\[
\xymatrix{ 
C^{-1} = A \ar[r]^{u^R}  \ar[rd]^{\overline{u}^R} & \End_B(M) =C^0\\
& \End_{\cH}(M) =\Tot(C^\bullet) \ar[u].
}
\]

\begin{proposition}\label{prop:descent}
The functor $\overline{u^R}$ has a fully faithful, continuous, monoidal left adjoint, $\overline{c}^L$, and the essential image of $\overline{c}^L$ is Morita equivalent to $\cH$. Moreover, if $F$ is conservative, then $\overline{F}$ is an equivalence.
\end{proposition}
\begin{proof}
The augmented cosimplicial object $\widetilde{C}^\bullet$ satisfies the Beck-Chevalley conditions. Thus, by Theorem \ref{thm:md}, $\overline{u}^R$ has a fully faithful left adjoint $\overline{c}^L$. Moreover, if $u^R$ is conservative, then $\overline{u}^R$ is an equivalence. 

The functor $\overline{c}^L$ is naturally monoidal, by virtue of being the fully faithful left adjoint to the monoidal map $\overline{u}^R$ ($\overline{c}^L$ is naturally oplax monoidal as the left adjoint to $\overline{u}^R$; as $\overline{u}^R\overline{c}^L \simeq id_{\End_B(M)}$, it follows that the monoidal structure of $\overline{c}^L$ is strong). 

It remains to show that $M$ defines a Morita equivalence between the essential image of $\overline{c}^L$ and $\cH$. Replacing $A$ by the essential image of $\overline{c}^L$, we may assume $u^R$ is conservative. Note that 
\begin{align*}
\ls \vee M \otimes_\cH M &\simeq | C_\bullet | \\
&\simeq \Tot(C^\bullet)\\
&\simeq A.
\end{align*}
Note also that the augmented cosimplicial category $\widetilde{C}^{\bullet}$ is naturally an augmented cosimplicial object of $A-A$-bimodules. In particular, the morphism $\overline{u}^R$ is a morphism of $A-A$-bimodules. As we have already noted $\cH \simeq M \otimes_A \ls \vee M$ as $\cH-\cH$-bimodules, and thus $A$ is Morita equivalent to $\cH$ as claimed.
\end{proof}

\begin{remark}
The Beck-Chevalley conditions for the augmented cosimplicial object $\widetilde{C}^{\bullet}$ imply a projection formula for the adjunction $u^R: A \leftrightarrows \End_B(M):c^L$
\[
c^L(u^R(a) \ast b) \simeq a \ast c^L(b).
\]
In particular, the object $\Spr_M := c^L(id_M) \in A$ is a coalgebra object which we call the \emph{Springer object}. \footnote{In the case $A= \cD(G)$, $M= \cD(G/N)$, and $B=\cD(H)$ (using the notiation of Section \ref{sec:proof}), the object $\Spr_M$ is the Springer sheaf.} The comonad $c^L u^R$ acting on $A$ is given by $a \mapsto \Spr_M \ast a$. Similarly, the object $e_M := \overline{c}^L(id_M)$ is an idempotent coalgebra object (it is the unit for the monoidal structure on the essential image of $\overline{c^L}$), which represents the idemotent comonad $\overline{c}^L\overline{u}^R$. In other words, the essential image of $\overline{c}^L$ is naturally identified with $e_M$-comodule objects in $A$.
\end{remark}

\section{Modules for the categorical group algebra}

\subsection{$G$-invariants and coinvariants}
In order to apply the results of the previous section, we need the following important consequence of Gaitsgory's 1-affineness theorem.

\begin{theorem}[\cite{1affine}, \cite{dario} Theorem 3.5.7]\label{theorem:D(G)andHC}
For any complex reductive group $G$, the $\cD(G)$-module $\cU(\fg)-\module$ defines a Morita equivalence between $\cD(G)$ and $\End_{\cD(G)}(\cU(\fg)) \simeq \HC_G$.
\end{theorem}
\begin{remark}
The functor $(-)\otimes_{\cD(G)} \left(\cU(\fg)-\module\right)$ from $\cD(G)-\module$ to $\HC_G-\module$ is known as the the functor of \emph{weak} $G$-coinvariants. Given a stack $X$ with an action of $G$, $\cD(X)$ is a  $\cD(G)$-module, and the corresponding $\HC_G$-module, $\cD(X) \otimes_{\cD(G)} \left(\cU(\fg)-\module\right)$, can be identifed with the category $\cD(\wq GX)$ of $G$-weakly equivariant objects of $\cD(X)$.  
\end{remark}

\begin{proposition}[\cite{dario}, section 2.5]\label{prop:rigidpivotal}
The monoidal category $\HC_G$ is rigid and pivotal.
\end{proposition}

Proposition \ref{prop:rigidpivotal} follows from the fact that there is a monoidal functor $\Rep(G) \to \HC_G$ with a continuous and conservative right adjoint, and the fact that $\Rep(G)$ itself is rigid and pivotal.

Given a $\cD(G)$-module $M$, let $M^G = \Hom_{\cD(G)}(\Vect, M)$ denote the strong $G$-invariants, and $M_G = \Vect \otimes_{\cD(G)} M$ the strong $G$-coninvariants.
\begin{proposition}[\cite{dario}, Theorem 3.6.1]
If $M$ is a $\cD(G)$-module, then $M_G \simeq M^G$.
\end{proposition}

Now let $X$ be a perfect stack with an action of $G$. We have the \v Cech simplicial stack with $n$-simplices $X\times G^{\times n}$ which is augmented by $X/G$. Taking categories of $D$-modules, we get a cosimplicial category with $n$-simplices $\cD(X\times G^{\times n})$, and structure maps of the form $f^!$ for each map $f$ in the Cech simplicial stack. After identifying $\cD(X\times G^{\times n})$ with $\cD(X) \otimes \cD(G)^{\otimes n}$, this cosimplicial category becomes identified with $\BAR^\bullet(\cD(G); \cD(X), \Vect)$. Thus, there is a natural equivalence $\cD(X)^G \simeq \cD(X/G)$. The result of Beraldo above implies that we can also identify the coinvariants $\cD(X)_G$ with $\cD(X/G)$. In fact, we have the following result.

\begin{proposition}[\cite{dario}]
\label{prop:tensorquotient}
Let $X$ be a perfect stack with an action of $G$. The canonical augmentations
\[
\cD(X/G) \xrightarrow{p^!} \BAR^\bullet(\cD(G);\cD(X), \Vect)
\]
and
\[
\BAR_\bullet(\cD(G);\cD(X), \Vect) \xrightarrow{p_\ast} \cD(X/G)
\]
realize $\cD(X/G)$ as the totalization and geometric realization respectively.
\end{proposition}

\begin{corollary}
\label{cor:hh}
Let $G$ be a reductive groups and $X$ and $Y$ perfect stacks with an action of $G$.  Then there is a canonical equivalence
\[
\cD(X) \otimes _{\cD(G)} \cD(Y) \simeq \cD\left((X \times Y)/G\right).
\]
\end{corollary}

\begin{lemma}\label{lemma:dual}
Let $G$ and $H$ be reductive groups and $X$ a perfect stack with an action of $G\times H$. Then $\cD(X)$ is $\cD(G)$ and $\cD(H)$ dualizable and self dual as a $\cD(H)-\cD(G)$-bimodule. Moreover, after identifying $\cD(X) \otimes_{\cD(H)} \cD(X)$ with $\cD(H \bs (X\times X))$ and $\cD(X) \otimes_{\cD(G)} \cD(X)$ with $\cD((X\times X)/G)$ using Corollary \ref{cor:hh}, the duality data
\[
\xymatrix{
\cD(G)   \ar@/^/[rr]^{u^R} & &\ar@/^/[ll]^{c^L} \cD(X) \otimes_{\cD(H)} \cD(X)\\
\cD(H)   \ar@/^/[rr]^{u^L} & &\ar@/^/[ll]^{c^R} \cD(X) \otimes_{\cD(G)} \cD(X)\\
}
\]
are given by:
\begin{align*}
u^L &= \delta_\ast \gamma^! \\
c^L &= \alpha_\ast \beta^! \\
u^R &= \beta_\ast \alpha^! \\
c^R &= \gamma_\ast \delta^!,
\end{align*}
where $\alpha, \beta, \gamma, \delta$ are the canonical maps in the following correspondences:
\begin{equation}\label{diagram:horocycle}
\xymatrix{
G & \ar[l]_\alpha  (H \bs X) \times G  \ar[r]^\beta & H\bs (X\times X) \\
H & \ar[l]_\gamma H \times (X/G) \ar[r]^\delta & (X \times X)/G.
}
\end{equation}
\end{lemma}
\begin{proof}
Suppose one takes (as an ansatz) the unit and counit maps to be as claimed above; it then follows from base change that the triangle identities are satisfied, and thus the ansatz unit and counit are indeed the unit and counit of an adjunction as required.
\end{proof}

\subsection{Categorical Highest Weight Theorem}
Fix a reductive group $G$, a Borel subgroup $B$, and take $N=[B,B]$, $H=B/N$, and $X=N\bs G$.
\begin{proposition}
The $\cD(H)-\cD(G)$-bimodule $\cD(N\bs G)$ is properly dualizable and self dual, and the action morphism $u^R$ is conservative.
\end{proposition}
\begin{proof}
The dualizability follows from Lemma \ref{lemma:dual}.  We have $u^L = \delta_\ast \gamma^!$, and $c^L = \alpha_\ast \beta^!$ (where $\alpha, \beta,\gamma, \delta$ are as in Diagram \ref{diagram:horocycle}). In this case, $\alpha$ and $\delta$ are proper, and $\beta$ and $\gamma$ are smooth, thus $\alpha_\ast$, $\delta_\ast$, $\beta^!$, and $\gamma^!$ all have continuous right adjoints as required.

It remains to show that $u^R$ is conservative. As observed by Mirkovic-Vilonen \cite{MVcharacter}, the composite $c^L u^R$ is the endofunctor on $\cD(G)$ given by convolution with the Springer sheaf $\cS_G := (\widetilde{\cN} \to \cN)_\ast \omega_{\widetilde{\cN}}$. As the Springer sheaf contains the unit object for convolution $\delta_e \in \cD(G)$ as a direct summand, we have that $c^L u^R$ contains the indentity functor as a summand. In particular, $u^R$ is conservative.
\end{proof}

\begin{corollary}
The bimodule $\cD(N\bs G)$ defines a Morita equivalence between $\cD(G)$ and $\cH= \cD(\quot NGN)$.
\end{corollary}

\begin{theorem}
The bimodule $\cD(\quot NGN)_H$ defines a Morita equivalence between $\cD(\quot NGN)$ and $\cD_H(\quot NGN)_H$.
\end{theorem}
\begin{proof}
The functor 
\[
M \mapsto M\otimes_{\cD(\quot NGN)} \cD(\quot NGN)_H.
\]
the functor of weak $H$-invariants. Thus the result follows from the 1-affineness theorem of Gaitsgory \cite{1affine}.
\end{proof}

\end{document}